\documentclass[
 hf,
]{ceurart}

\usepackage{listings}
\usepackage{microtype}

\usepackage{amsmath}
\usepackage{amsthm}
\usepackage{amssymb}
\usepackage{mathtools}
\usepackage{thmtools, thm-restate}

\usepackage[english]{babel}

\usepackage{tikz}
	\usetikzlibrary{matrix,arrows}

\usepackage{graphicx} 
\usepackage{subcaption}

\usepackage[inline]{enumitem}
\usepackage{hyperref}
    \hypersetup{
        colorlinks  	= true, 
        allcolors   	= blue, 
    	pdfstartview	= Fit,
        breaklinks		= true
}
\usepackage[all]{hypcap} 

\declaretheorem[style=plain, sibling=theorem]{lemma, proposition, corollary, fact}
\declaretheorem[style=definition]{definition, construction, example, assumption}
\declaretheorem[style=remark]{remark, note}

\newcommand{\type}[1]{\mathit{type}_{#1}}
\newcommand{\src}[1]{\mathit{src}_{#1}}
\newcommand{\tar}[1]{\mathit{tar}_{#1}}
\newcommand{\TG}{\mathit{TG}}
\newcommand{\incl}[2]{\iota^{#1}_{#2}}

\newcommand{\Sub}[1]{\mathbf{Sub}(#1)}
\newcommand{\GraphTG}{\mathbf{Graph}_{\mathbf{TG}}}

\newcommand{\Mcal}{\mathcal{M}}

\newcommand{\Ical}{\mathcal{I}}
\newcommand{\Imodels}{\models_{\Ical}}
\newcommand{\nImodels}{\nvDash_{\Ical}}
\newcommand{\Iequiv}{\equiv_{\Ical}}

\newcommand{\Jcal}{\mathcal{J}}
\newcommand{\qm}{q_{\mathrm{m}}}
\newcommand{\qe}{q_{\mathrm{e}}}
\newcommand{\qmSPI}{\qm^{T}}
\newcommand{\gm}{g_{\mathrm{m}}}
\newcommand{\gep}{g_{\mathrm{e}}}
\newcommand{\gmSPI}{\gm^{T}}

\DeclareMathOperator{\Inst}{Inst}
\DeclareMathOperator{\Flat}{Flat}
\DeclareMathOperator{\Ima}{Im}
\DeclareMathOperator{\nl}{nl}

\newcommand{\Cmax}{C_{\mathrm{max}}}

\newcommand{\clb}{c_{\mathrm{lb}}}
\newcommand{\cub}{c_{\mathrm{ub}}}

\newcommand{\clbSubT}[1]{c_{\mathrm{lb}}^{\Sub{#1}}}
\newcommand{\cubSubT}[1]{c_{\mathrm{ub}}^{\Sub{#1}}}

\newcommand\isomto{\stackrel{\textstyle\sim}{\smash{\longrightarrow}\rule{0pt}{0.4ex}}}

\graphicspath{ {./figures/} } 

\newif\iflong
\longtrue

\iflong
	\newcommand{\shortorlong}[2]{#2}
\else
	\newcommand{\shortorlong}[2]{#1}
\fi

\begin{document}
    \copyrightyear{2026}
    \copyrightclause{Copyright © 2026 for this paper by its authors. Use permitted under Creative Commons License Attribution 4.0 International (CC BY 4.0)}
    \conference{Joint Proceedings of the STAF 2026 Workshops: AgileMDE, GCM, ICMM, LLM4SE, TTC. Rennes, France, June 29--July 3, 2026}
    
    \iflong
	   \title{A Nesting-Free Normal Form for Nested Conditions in Finite Lattices of Subgraphs -- Extended Version}
    \else
    	\title{A Nesting-Free Normal Form for Nested Conditions in Finite Lattices of Subgraphs}
    \fi

    \author[1]{Jens Kosiol}[%
        orcid=0000-0003-4733-2777,
        email=kosiolje@mathematik.uni-marburg.de,
    ]
    \address[1]{Philipps-Universität Marburg, Marburg, Germany}

    \author[2]{Steffen Zschaler}[%
        orcid=0000-0001-9062-6637,
        email=szschaler@acm.org,
    ]
    \address[2]{King\rq{}s College London, London, UK}

    \begin{abstract}
        We present a nesting-free normal form for the formalism of \emph{nested conditions and constraints} in the context of finite lattices of subgraphs. 
    \end{abstract}


    \maketitle

    \section{Introduction}
    \label{sec:introduction}
    In this paper, we present a nesting-free normal form for the formalism of \emph{nested conditions and constraints}~\cite{Rensink04, HP05} in the context of finite lattices of subgraphs. 
Nested conditions express properties of morphisms, while constraints express properties of objects. 
This formalism has established itself as the central logical formalism to be used in the context of \emph{graph transformation} because it allows to elegantly compute the interactions of the applications of transformation rules and the validity of constraints, e.g.,~\cite{HP09, EGHLO14, EhrigGHLO12, GHE14, EhrigEGH15, KosiolSTZ22, LauerKT25, KonigRSU25}. 
On (typed) graphs, nested constraints have been shown to be expressively equivalent to first-order logic~\cite{Rensink04, HP09}. 
\emph{Moreover, every additional level of nesting increases the expressivity of the formalism~\cite{Rensink04}.} 

While predominantly used for typed graphs, the logic of nested conditions and constraints can be used in any category. 
In ongoing work~\cite{KZ26}, we employ this formalism in the context of finite lattices of subgraphs, i.e., categories whose objects are the (equivalence classes of) subgraphs of a given finite \emph{container graph} $T$ and whose morphisms are inclusions; we denote such a category as $\Sub{T}$. 
Transformations of these objects can be realised via so-called \emph{subgraph transformation systems}~\cite{CHS08, HCE14}, where transformation rules themselves consist of subgraphs of the given container graph $T$ and are applied to subgraphs of $T$ in such a way that the result is again a subgraph of $T$. 

In this paper, we investigate the formalism of nested conditions in the context of finite lattices of subgraphs and make two contributions.
First, we work out a normalisation construction, called \emph{flattening}, that transforms every nested condition/constraint into a \emph{nesting-free} Boolean combination of literals (i.e., graph patterns). 
In that respect, our central result is the correctness of that construction (Theorem~\ref{thm:semantics-flattening}), implying that every condition/constraint can be equivalently transformed into a conjunctive normal form on the basis of literals (Corollary~\ref{cor:normal-forms-exist}). 
We present our flattening construction in Section~\ref{sec:sub_t_constraints}.

Furthermore, every graph $G$ in a lattice of subgraphs of a container graph $T$ can also be considered as a graph in the category $\GraphTG$, the category of all graphs typed over $\TG$, the type graph over which the container graph $T$ is typed. 
In particular, a nested constraint formulated in the category $\GraphTG$ can be evaluated for subgraphs of $T$. 
We develop a translation that, given a nested condition/constraint for $\GraphTG$ and a \emph{finite} container graph $T$ (typed over $\TG$) transforms that condition/constraint into one formulated in the category $\Sub{T}$.
Our main contribution in that regard is showing our translation to preserve the semantics of the given constraints (Theorem~\ref{thm:instantiation-preserves-semantics}). 
This ensures that, while nesting of constraints collapses in the context of lattices of subgraphs, no expressivity of the formalism is lost given a finite container graph $T$. 
Moreover, this translation makes specifying and working with nested conditions/constraints in the context of finite lattices of subgraphs much more practical, since nested conditions/constraints in $\GraphTG$ are much more compact than their counterparts in $\Sub{T}$.
We present our translation in Section~\ref{sec:reduction-constraints-to-sub-t}. 

The results in this paper are not surprising. 
It is intuitive that in a context where all finitely many (graph) elements that can ever appear are known beforehand, properties of these graphs can be expressed as a Boolean combination of non-nested graph patterns that list the different ways in which the property in question can (should not) hold. 
Still, making this precise is somewhat subtle. 
Moreover, in~\cite{KZ26}, we develop an interesting application of the results presented here. 
Considerably generalising work of Horcas et al.~\cite{HSBMZ22}, we there present an approach to \emph{non-blocking consistency-preserving subgraph transformation}, where, from a given subgraph transformation rule, we derive a family of rules that (i) contain the given rule as subrule, (ii) cannot introduce violations of the given constraints upon application and (iii) are not blocked by application conditions. 
Instead, the original rule is amended by additional actions that immediately remedy introduced constraint violations. 
We see considerable potential of this rule construction in particular in the context of \emph{model-driven optimisation}~\cite{JKLT23}. 
Importantly, the rule construction we develop in~\cite{KZ26} is based on the normal form for conditions in $\Sub{T}$ we develop here.

\paragraph{Related Work}
Nested conditions and constraints have been introduced by Rensink~\cite{Rensink04} and by Habel and Pennemann~\cite{HP05} to increase the expressiveness of non-nested constraints and application conditions that had been considered before~\cite{EEPT06}. 
Nested conditions are expressively equivalent to first-order logic on graphs~\cite[Theorems~1 and~3 or Theorems~10 and~11, respectively]{Rensink04, HP09}. 
Searching for normal forms of formulas is a standard technique in logic (see, e.g.,~\cite{Avigad22} for classic propositional and first-order logic). 
On graphs, every level of nesting increases expressiveness of nested conditions~\cite[Theorem~4]{Rensink04}; consequently, normal forms for this formalism cannot generally avoid nesting but recursively require the individual levels of nesting to be in a certain normal form like CNF (see, e.g.,~\cite[ch.~6.3]{Pennemann2009}). 
Furthermore, there exist normal forms for nested constraints that require quantifiers to alternate with nesting level. 
These, however, either need to allow for isomorphisms in the conditions~\cite{StoltenowKSCLO24} or are only guaranteed to exist for \emph{linear} conditions (i.e., conditions without conjunctions and disjunctions)~\cite{SH19}. 
In this paper, we develop a non-nested normal form for nested conditions in the very restricted setting of the lattice of subgraphs of a given graph.

\paragraph{Structure}
Before presenting our contribution in Sections~\ref{sec:sub_t_constraints} and~\ref{sec:reduction-constraints-to-sub-t}, we begin by introducing the \emph{class--responsibility assignment (CRA) problem} as a running example we use to illustrate our constructions (Section~\ref{sec:running-example}) and by presenting preliminaries (Section~\ref{sec:preliminaries}). 
We conclude in Section~\ref{sec:conclusion} and provide all proofs \shortorlong{in an extended version of this paper~\cite{KosiolZ26}}{in Appendix~\ref{app:proofs}}.
\emph{Sections~\ref{sec:running-example} and~\ref{sec:preliminaries} are taken over from~\cite{KZ26} with only minor adaptations}.

    \section{Running Example: The CRA Case}
    \label{sec:running-example}
    In the CRA problem~\cite{Fleck+16}, a set of methods and attributes must be allocated to a (to be determined) number of classes to optimise cohesion and coupling metrics of the resulting software design.
It is not immediately obvious that the search space for this problem is finite, as it might involve the generation of an unspecified number of new classes (which are nodes in a graph representing allocations of methods/attributes to classes).
However, once one realises that any meaningful solution contains at most as many classes as there are features (methods and attributes), it becomes clear that the problem is about determining the correct set of edges between features and classes. 
In particular, we will be able to understand the solutions of a given problem instance as a lattice of subgraphs $\Sub{T}$. 

\begin{figure}
    \centering
    \includegraphics{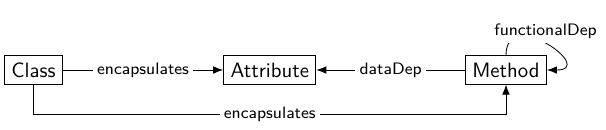}
    \caption{A meta-model (type graph) $\TG$ for defining instances of the CRA problem (based on~\cite{Fleck+16})}
    \label{fig:MM-CRA}
\end{figure}

Figure~\ref{fig:MM-CRA} shows the metamodel (type graph) for the CRA problem---a class diagram whose instances describe specific sets of classes, attributes, and methods for which to determine an allocation.
Figure~\ref{fig:CRA-example-instance} then shows a model that instantiates the meta-model in Figure~\ref{fig:MM-CRA} and describes a concrete problem instance $P$ with 3 methods and 3 attributes that need to be allocated to classes.
Figure~\ref{fig:CRA-example-solution} shows a corresponding instance $S$ of the metamodel that represents one possible solution to the problem instance defined by $P$.
Here, we have introduced two classes and made specific assignments of methods and attributes to these classes. 
The distinction between \emph{problem instances} and \emph{solutions} is not of importance for this paper; we have discussed this distinction in more detail in~\cite{KZ26}, where it becomes relevant, and we will shortly comment on it in the following section. 

\begin{figure}
    \centering
    \includegraphics[width=\textwidth]{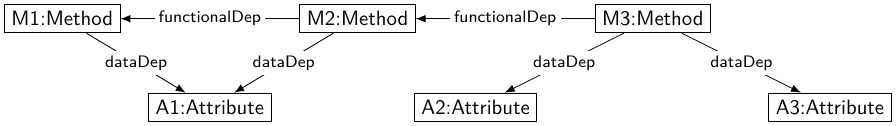}
    \caption{Example instance $P$ of the CRA problem}
    \label{fig:CRA-example-instance}
\end{figure}

\begin{figure}
    \centering
    \includegraphics[width=\textwidth]{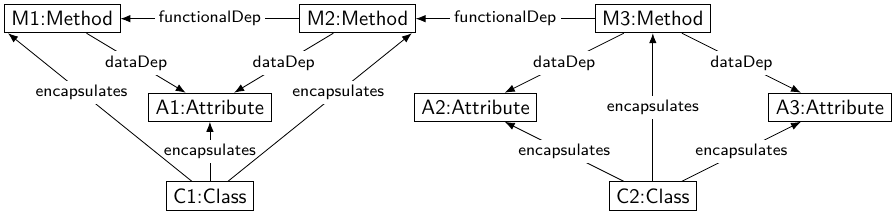}
    \caption{Example of a solution $S$ to the CRA problem instance depicted in Fig.~\ref{fig:CRA-example-instance}}
    \label{fig:CRA-example-solution}
\end{figure}

    \section{Preliminaries}
    \label{sec:preliminaries}
    In this section, we introduce our preliminaries. 
As our general background, we recall \emph{typed graphs}~\cite{EEPT06} and \emph{nested conditions and constraints}~\cite{Rensink04, HP09}. 
Moreover, we present \emph{finite lattices of subgraphs}, for which we investigate nested conditions and constraints more closely in this paper. 

\subsection{Typed Graphs}

	Graphs are sets of nodes with edges between them. 
	Typing offers a way to equip the individual elements of a graph with meaning; a \emph{type graph} provides the available types for elements, and morphisms assign the elements of \emph{typed graphs} to those types. 

	\begin{definition}[Graph. Typed graph]
		A \emph{graph} $G = (V_G,E_G,\src{G},\tar{G})$ consists of a set of nodes $V_G$, a set of edges $E_G$, and source and target functions $\src{G},\tar{G}\colon E_G \to V_G$. 
		A \emph{graph morphism} $f \colon G \to H$ between two graphs is a pair of functions $f_V\colon V_G \to V_H$ and $f_E\colon E_G \to E_H$ that is compatible with the source and target functions, such that 
			$\src{H} \circ f_E = f_V \circ \src{G} \text{ and } \tar{H} \circ f_E = f_V \circ \tar{G}$.
		
		A \emph{type graph} $\TG= (V_{\TG},E_{\TG},\src{\TG},\tar{\TG})$ is a fixed graph. 
		A \emph{graph typed over a fixed type graph $\TG$} consists of a graph $G$ and a graph morphism $\type{G} \colon G \to \TG$. 
		A \emph{typed graph morphism} $f\colon G \to H$ between graphs $G$ and $H$ typed over $\TG$ is a graph morphism that is compatible with typing, i.e., such that $\type{H} \circ f = \type{G}$. 
		
		A typed graph morphism is \emph{injective/surjectiv/bijective/an inclusion} if both of its components are.
		We denote injective morphisms via hooked arrows $f\colon G \hookrightarrow H$ and regularly denote inclusions as $\incl{G}{H}\colon G \subseteq H$, where the indices denote the involved graphs. 

        A typed graph $G$ is \emph{finite} if both its sets of nodes $V_G$ and of edges $E_G$ are.
	\end{definition}

	\begin{remark}
		To ease language and notation, we adopt the following conventions and assumptions. 
		\begin{itemize}[nosep]
			\item In this paper, we work with typed graphs and always assume a fixed type graph $\TG$ over which all graphs are typed. 
			Except for the type graph, we will just speak of \emph{graphs}, and we do not introduce a special notation for typed graphs. 
			Thus, speaking of a graph $G$, we tacitly assume a type graph $\TG$ and a type morphism $\type{G} \colon G \to \TG$ to exist. 
			Similarly, we also regularly just speak of morphisms (and not of typed ones).
			
			\item Via $\emptyset$ we denote the \emph{initial} or \emph{empty graph}, also when it is considered to be typed over the given type graph $\TG$ (via the empty function).
			\item By $\GraphTG$ we denote the category of all graphs typed over the given type graph $\TG$. 
		\end{itemize}%
	\end{remark}

	\begin{example}\label{ex:type-graph-and-typed-graph}
       The graph $\TG$ from Fig.~\ref{fig:MM-CRA} can serve as a type graph, modelling the domain of the CRA problem. 
       The nodes provide the types \textsf{Class, Attribute} and \textsf{Method}. 
       \textsf{Methods} can have dependencies on \textsf{Methods} (\textsf{functionalDep}) and on \textsf{Attributes} (\textsf{dataDep}). 
       \textsf{Classes} can \textsf{encapsulate} \textsf{Methods} and \textsf{Attributes}.

       The further graphs from Sect.~\ref{sec:running-example}, namely graph $P$ (Fig.~\ref{fig:CRA-example-instance}) and graph $S$ (Fig.~\ref{fig:CRA-example-solution}), are typed over that type graph $\TG$ (and the typing morphisms are indicated by the annotations of nodes and edges). 
	\end{example}

\subsection{Lattices of Subgraphs}

	\emph{Subobject transformation systems} have been developed as a tool to study different notions of causality between rule applications in (double-pushout) rewriting systems~\cite{CHS08,HCE14}. 
	Intuitively, one studies transformations of subobjects of a given object from an ($\Mcal$-)adhesive category. 
    In this paper, we do not care about this notion of transformation but more closely investigate nested conditions for the categories for which these transformations are defined. 
    In that, we restrict ourselves to the case where underlying objects are graphs. 
	For simplicity of presentation, we work with subgraphs instead of equivalence classes of injective morphisms. 

	\begin{definition}[Subgraph. (Finite) lattices of subgraphs]
		Given a typed graph $G$, a \emph{subgraph} of it is a typed graph $H$, such that $V_H \subseteq V_G$, $E_H \subseteq E_G$ and $\src{H}$, $\tar{H}$, and $\type{H}$ are restrictions of the according functions of $G$ (in particular, with every edge $e \in E_H$ its source and target nodes in $V_G$ belong to $V_H$). 
		Via $H \subseteq G$ we denote the subgraph relationship between graphs. 
		
		Given a (finite) typed graph $T$, the \emph{(finite) lattice of subgraphs of~$T$}, denoted as $\Sub{T}$, has graphs $G\subseteq T$ as objects and inclusions $\incl{G}{H}\colon G\subseteq H$ between such subgraphs as morphisms.
		We call $T$ the \emph{container graph} of that category. 
	\end{definition}

	\begin{remark}
		Finiteness of the container graph $T$ obviously implies finiteness of $\Sub{T}$. 
        What we call the \emph{container graph} would be called \emph{type graph} in the terminology of~\cite{CHS08}. 
		We deviate from that terminology to distinguish the container graph $T$ from the type graph $\TG$ over which $T$ (and all graphs from $\Sub{T}$) are typed.
	\end{remark}

	\begin{example}\label{ex:container-and-subgraph} 
         \begin{figure}
            \centering
            \includegraphics[width=\textwidth]{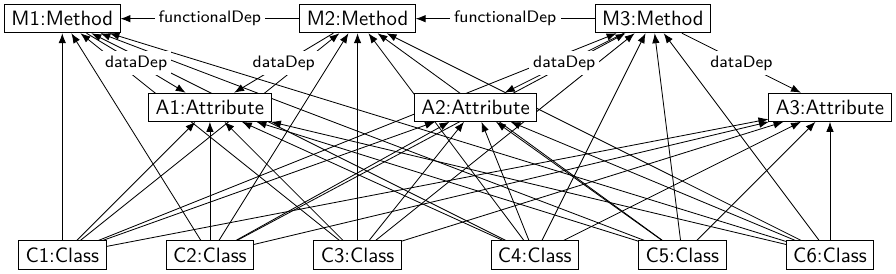}
            \caption{%
				A container graph $T$ defining a lattice of subgraphs $\Sub{T}$ that includes all interesting solutions to the CRA problem instance $P$ defined in Fig.~\ref{fig:CRA-example-instance}. 
        		In this graph, we omit the \textsf{encapsulates} annotation to not further clutter the figure.%
			}
            \label{fig:CRA-example-solution-T}
        \end{figure}

        Given the type graph $\TG$ from Fig.~\ref{fig:MM-CRA} (see Sect.~\ref{sec:running-example} and Example~\ref{ex:type-graph-and-typed-graph}), also the graph $T$ depicted in Fig.~\ref{fig:CRA-example-solution-T} is typed over it.
        Considering the CRA problem instance defined by the graph $P$ shown in Fig.~\ref{fig:CRA-example-instance}, $\Sub{T}$ defines a lattice that contains all interesting solutions to that problem instance (disregarding solutions with more than six \textsf{Classes} of which several will be necessarily empty). 
        The graph $S$ depicted in Fig.~\ref{fig:CRA-example-solution} is one of the elements of $\Sub{T}$. 
        Note that $\Sub{T}$ is actually too large to be considered the search space for finding solutions for the problem instance defined by the graph $P$: we are only interested in graphs also containing that instance $P$; otherwise, the problem definition would be changed (containing a different set of \textsf{Methods} and \textsf{Attributes} and/or other dependencies). 
        In~\cite{KZ26}, we address that problem by only considering graphs $G \in \Sub{T}$ for which $P \subseteq G$, where $P \subseteq T$ is considered to define a problem instance; a category that we denote as $\Sub{T \supseteq P}$. 
        In this paper, however, we can safely ignore this subtlety.
	\end{example}

    Categories $\Sub{T}$ consisting of the subgraphs of $T$ actually constitute lattices~\cite{CHS08, HCE14}:
	\begin{remark}
		In any category of subgraphs $\Sub{T}$, if there exists a morphism between subgraphs $A$ and $B$, this morphism is unique (by left-cancelability of monomorphisms). 
        Moreover, given an adhesive category $\mathcal{C}$ (such as the category of graphs typed over a given type graph $\TG$), for every object $T$ of $\mathcal{C}$, $\Sub{T}$ constitutes a distributive lattice~\cite[Theorem~5.1 and Corollary~5.2]{LS05}.
		Given two subobjects $A,B \subseteq T$, their \emph{meet} $A \cap B$ in $\Sub{T}$ arises from pulling back $A \to T \leftarrow B$ in $\mathcal{C}$ and their \emph{join} $A \cup B$ in $\Sub{T}$ from pushing out $A \leftarrow A \cap B \to B$ in $\mathcal{C}$. 
		In our setting, where we work with subgraphs instead of equivalence classes of injections, this means that the meet $A \cap B$ is literally the intersection of $A$ and $B$ while the join $A \cup B$ is their union.  
	\end{remark}

\subsection{Nested Conditions and Constraints}
\label{sec:preliminaries:conditions_and_constraints}

	Next, we define \emph{nested conditions} and \emph{constraints}~\cite{HP09,Rensink04,EGHLO14}. 
	Nested conditions provide a logic to express properties of morphisms in a category; constraints express properties of objects. 
	For (typed) graphs, nested conditions are expressively equivalent to standard first-order logic on graphs~\cite{Rensink04,HP09},
	and every level of nesting increases the expressivity~\cite{Rensink04}. 

	In this work, given a type graph $\TG$ and a (finite) graph $T$ typed over $\TG$, we consider two different categories: 
	The category of graphs typed over $\TG$, $\GraphTG$, and the category $\Sub{T}$ of subgraphs of $T$ (with inclusions as morphisms). 
	Since every graph from $\Sub{T}$ is also a graph typed over $\TG$ (by composition of the inclusion with the typing of $T$), we obtain two variants of defining formulas for it.
	We define conditions and constraints category-theoretically and afterwards illustrate the definitions for $\GraphTG$ and $\Sub{T}$.
	\begin{definition}[(Nested) conditions and constraints]\label{def:conditions-and-constraints} 
		In a category $\mathcal{C}$, given an object $C_0$, a \emph{(nested) condition} over $C_0$ is defined recursively as follows:
		\begin{itemize}
			\item \textsf{true} is a condition over $C_0$.
			\item If $a_1\colon C_0 \hookrightarrow C_1$ is a monomorphism and $d$ is a condition over $C_1$, $\exists \, (a_1\colon C_0 \hookrightarrow C_1, d)$ is a condition over $C_0$. 
			\item Moreover, Boolean combinations of conditions over $C_0$ are conditions over $C_0$. 
		\end{itemize}
		In a category with a strict initial object $\emptyset$, a \emph{(nested) constraint} is a condition over $\emptyset$. 
		
		We recursively define the \emph{nesting level} of a condition $c$ as
		\begin{align*}
			\nl(\textsf{true}) & = 0; & 
            \nl(\exists(a_1\colon C_0 \hookrightarrow C_1,d)) & = \nl(d) + 1; \\
			\nl(\neg d) & = \nl(d); & 
			\nl(d_1 \odot d_2) & = \max\{\nl(d_1),\nl(d_2)\}\text{, where } \odot\in\{\vee,\wedge\} .
		\end{align*}
		
		\emph{Satisfaction} of a nested condition $c$ over $C_0$ for a morphism $g\colon C_0 \rightarrow G$, denoted as $g \models c$, is defined as follows: 
		\begin{itemize}
			\item Every morphism satisfies \textsf{true}.
			\item The morphism $g$ satisfies a condition of the form $c = \exists \, (a_1\colon C_0 \hookrightarrow C_1, d)$ if there exists a monomorphism $q\colon C_1 \hookrightarrow G$ such that $g = q \circ a_1$ and $q \models d$.
			\item For Boolean operators, satisfaction is defined as usual.
		\end{itemize}
		An object $G$ satisfies a constraint $c$, denoted as $G \models c$, if the unique monomorphism $\iota_{G}\colon \emptyset \hookrightarrow G$ does so. 
	\end{definition}

	\begin{remark}
		In the notation of conditions, we drop the domains of the involved morphisms and occurrences of \textsf{true} whenever they can be unambiguously inferred. 
		Moreover, we introduce $\forall \, (C_1, d)$ as an abbreviation for the condition $\neg \exists \, (C_1, \neg d)$. 
		For example, we write $\forall \, (C_1, \exists \, C_2)$ instead of $\neg \exists \, (a_1\colon \emptyset \hookrightarrow C_1, \neg \exists \, (a_2\colon C_1 \hookrightarrow C_2, \textsf{true}))$. 
        At other places, especially where it is convenient for compact presentation of computations, we will not display the graphs of a condition but merely the identifiers of the involved morphisms, writing, e.g., $\forall \, (a_1, \exists \, a_2)$.  
		
		To be able to notationally distinguish between satisfaction of constraints and conditions in the category $\GraphTG$ from the one in $\Sub{T}$ (where $T$ is typed over $\TG$), we use the symbol $\Imodels$ to denote satisfaction in $\Sub{T}$ and $\Iequiv$ to denote equivalence in $\Sub{T}$. 
        Moreover, we denote the morphisms in conditions over graphs from $\Sub{T}$ as $b_i\colon B_{i-1} \subseteq B_i$ to visually differentiate these from (i) morphisms $a_i\colon C_{i-1} \hookrightarrow C_i$ in conditions over a graph from $\GraphTG$ and (ii) from inclusions $\incl{C_i}{G}\colon C_i \subseteq G \in \Sub{T}$ we use to check the satisfaction of such conditions. 
	\end{remark}

	\begin{example}\label{ex:constraintsGraphTGVsSubT}
        The definition of the CRA problem comes with additional constraints (compare Sect.~\ref{sec:running-example}): 
        As a minimum, one requires that every \textsf{Attribute} and every \textsf{Method} is encapsulated by exactly one \textsf{Class}. 

        \begin{figure}
            \centering
            \includegraphics{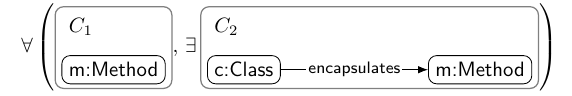}
            \caption{Constraint $\clb$ in $\GraphTG$ requiring every \textsf{Method} to be assigned to at least one \textsf{Class}}
            \label{fig:CRA-constraint-lowerbound}
        \end{figure}

	    \begin{figure}
            \centering
            \includegraphics{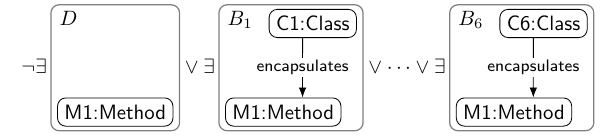}
            \caption{One conjunct of constraint $\clbSubT{T}$ in $\Sub{T}$, requiring \textsf{Method M1} to be assigned to at least one of \textsf{Classes C1, \dots, C6} from $T$ (in case it exists)}
            \label{fig:CRA-constraint-lowerbound-SubT-unsimplified}
		\end{figure}

        Figure~\ref{fig:CRA-constraint-lowerbound} depicts in $\GraphTG$ the constraint expressing the \emph{lower bound} on encapsulation of \textsf{Methods}, namely stating that every \textsf{Method} shall be encapsulated by (at least) one \textsf{Class} (the names of the elements in the graphs indicate the morphism from $C_1$ to $C_2$); in the following, we refer by $\clb$ to that constraint. 
        In $\Sub{T}$, in contrast, we can only express properties of concrete elements (as the constraints in $\Sub{T}$ are defined in terms of subgraphs of $T$); thus, the constraint in Fig.~\ref{fig:CRA-constraint-lowerbound-SubT-unsimplified} expresses that the concrete \textsf{Method M1} is encapsulated by \emph{a} \textsf{Class}---either \textsf{C1, C2, \dots} or \textsf{C6}. 
        The same constraint has to be added for \textsf{Methods M2} and \textsf{M3}. 
        By $\clbSubT{T}$, we refer to the conjunction of these constraints. 

        \begin{figure}
            \centering
            \includegraphics{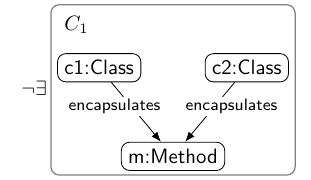}
            \caption{Constraint $\cub$ in $\GraphTG$ forbidding any \textsf{Method} to be assigned to more than one \textsf{Class}}
            \label{fig:CRA-constraint-upperbound}
        \end{figure}

	    \begin{figure}
            \centering
            \includegraphics{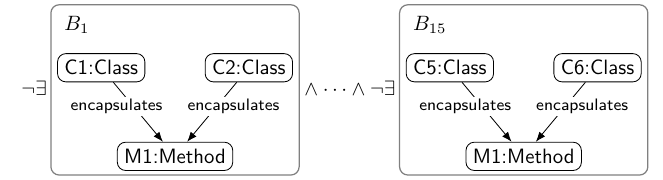}
            \caption{Conjunct of constraint $\clbSubT{T}$ in $\Sub{T}$, forbidding \textsf{Method M1} to be assigned to more than one of the \textsf{Classes C1, \dots, C6} from $T$}
            \label{fig:CRA-constraint-upperbound-SubT}
		\end{figure}

        Continuing, the constraint from Fig.~\ref{fig:CRA-constraint-upperbound} expresses in $\GraphTG$ that no \textsf{Method} is encapsulated by more than one \textsf{Class}, that is, it states an \emph{upper bound} on these encapsulations; in the following, we refer by $\cub$ to that constraint. 
        Again, in $\Sub{T}$, we have to express this for every \textsf{Method} individually and exclude any possible pair of \textsf{Classes}; the resulting set of constraints for \textsf{Method M1} is shown in Fig.~\ref{fig:CRA-constraint-upperbound-SubT}.
        By $\cubSubT{T}$, we refer to the conjunction of these constraints. 
	\end{example}

    \section{A Normal Form for Conditions in Lattices of Subgraphs}
    \label{sec:sub_t_constraints}
    In this section, we introduce a normal form for conditions in lattices of subgraphs $\Sub{T}$. 
In Sect.~\ref{sec:preliminaries:conditions_and_constraints}, we have already discussed that, given a type graph $\TG$ and a (finite)\footnote{%
    The results in this section actually also hold when one allows graphs to be infinite, in contrast to our results in Sect.~\ref{sec:reduction-constraints-to-sub-t}.%
} 
graph $T$ typed over it, we can consider conditions and constraints in two different categories: the category of graphs typed over $\TG$, $\GraphTG$, and the category $\Sub{T}$ of subgraphs of $T$. 
Working with conditions in $\Sub{T}$ facilitates the construction of activation diagrams and subsequent derivation of non-blocking consistency-preserving rules we develop in~\cite{KZ26}. 
In particular, these constructions depend on the normal form we develop in this section. 
However, conditions in $\Sub{T}$ are inconvenient for the specification of general constraints (because of their expressing properties of concrete elements of $T$). 
To address this, in Sect.~\ref{sec:reduction-constraints-to-sub-t}, we provide a way to \emph{instantiate} a condition defined for graphs from $\GraphTG$ into a condition for graphs from $\Sub{T}$ such that its semantics is preserved. 

The key feature of the normal form we construct in this section is that it avoids any nesting. 
We show that in $\Sub{T}$ for every condition there is an equivalent condition without nesting. (This is in stark contrast to the case of $\GraphTG$, where every level of nesting increases the expressivity~\cite{Rensink04}.) 
We start by defining \emph{literals}, the building blocks of our normal form. 
\begin{definition}[Literal]
    Given a type graph $\TG$ and a graph $T$ typed over it, a \emph{literal} in $\Sub{T}$ is a condition of the form \textsf{true}, \textsf{false}, $\exists \, (b\colon B_0 \subseteq B_1, \textsf{true})$ or $\neg \exists \, (b \colon B_0 \subseteq B_1, \textsf{true})$, where $B_0,B_1 \in \Sub{T}$. 
\end{definition}

Next, we develop our \emph{flattening construction} that simultaneously pushes negations inward and removes all nesting in a condition in $\Sub{T}$. 
For two positive existential quantifiers it applies the well-known equivalence $\exists \, (b_1, \exists \, (b_2,d)) \equiv \exists \, (b_2 \circ b_1,d)$~\cite[Fact~4]{HP05}. 
A positive existential condition that is followed by a negative one---that is, a condition of the form $\exists \, (b_1\colon B_0 \subseteq B_1, \neg d)$---is replaced by $\exists \, (b_1\colon B_0 \subseteq B_1) \wedge \neg \Tilde{d}$, where $\Tilde{d}$ is an adaptation of $d$ so that it becomes a condition over $B_0$ (instead of $B_1$). 
In $\GraphTG$, this generally is not a semantic equivalence because the first variant forbids an occurrence of $d$ at a fixed location for $b_1$ whereas the second variant forbids any occurrence of $\Tilde{d}$. 
For satisfaction in $\Sub{T}$, however, both notions coincide because there can only be one inclusion of a graph into another one. 
This is why nesting collapses in $\Sub{T}$.
In the following, we first prove the just stated equivalence and then use it in our \emph{flattening construction}, of which it is the crucial step to transform conditions in $\Sub{T}$ into a non-nested normal form. 

\begin{restatable}[Correctness of extracting negation]{lemma}{extractingNegation}\label{lem:extracting-negation}
    Given a type graph $\TG$ and a graph $T$ typed over it, for any condition $d$ in $\Sub{T}$ over a graph $B_1 \subseteq T$, it holds that
    \begin{equation*}
        \exists \, (b_0\colon B_0 \subseteq B_1, \neg d) \Iequiv \exists \, (b_0 \colon B_0 \subseteq B_1, \textsf{true}) \wedge \neg \exists \, (b_0 \colon B_0 \subseteq B_1, d) . \qedhere
    \end{equation*}
\end{restatable}

The lemma just stated is the crucial ingredient to proving that our following \emph{flattening construction} preserves semantics (in $\Sub{T}$). 
In flattening, we push negations over Boolean operators and extract them from existentially bound conditions; moreover, we compose existentially bound conditions. 
We end up with a condition that is nesting-free and where every negation either appears in front of an existentially bound condition or at nesting level zero (where the root of the condition can be some Boolean combination of \textsf{true}). 
We define flattening \emph{over a morphism}, where we use that morphism to collect all morphisms we have to compose in the resulting condition (that somewhat eases notation). 
Flattening of a condition can then just be performed by flattening that condition over the identity morphism of the graph over which the condition is defined.

\begin{construction}[Flattening]\label{constr:flattening}
    Given a type graph $\TG$, a graph $T$ typed over it, a graph $B_0 \in \Sub{T}$, a morphism $b_0$ with co-domain $B_0$ and a condition $c$ over $B_0$ in $\Sub{T}$, we recursively define to \emph{flatten} $c$ over $b_0$ via:
    \begin{align}
        \Flat(b_0, \textsf{true}) = \exists \, (b_0, \textsf{true})\tag{F1}\\
        \Flat(b_0, \neg\textsf{true}) = \exists \, (b_0, \textsf{false}) \equiv \textsf{false}\tag{F2}\\
        \Flat(b_0, d_1 \odot d_2) = \Flat(b_0, d_1) \odot \Flat(b_0, d_2)\tag{F3}\\ 
        \Flat(b_0, \neg (d_1 \odot d_2)) = \Flat(b_0, \neg d_1) \,\overline{\odot}\, \Flat(b_0, \neg d_2)\tag{F4}\\
        \Flat\left(b_0, \exists \, (b_1\colon B_0 \subseteq B_1, d)\right) = \Flat(b_1 \circ b_0, d)\tag{F5}\\
        \Flat(b_0, \neg \exists \, (b_1\colon B_0 \subseteq B_1, d) = \exists \, (b_0, \textsf{true}) \wedge \neg \Flat(b_1 \circ b_0, d)\tag{F6}\\
        \Flat(b_0, \neg \neg c) = \Flat(b_0, c)\tag{F7}
    \end{align}
    where $\odot \in \{\wedge, \vee\},\ \overline{\wedge} = \vee, \text{ and }\overline{\vee} = \wedge$.
\end{construction}
    
\begin{example}\label{ex:example-flattening}
    One regularly encounters constraints of the structure $\forall \, (b_0\colon \emptyset \hookrightarrow C_1, \exists \, (b_1\colon C_1 \hookrightarrow C_2, \textsf{true}))$ in practice---compare, for instance, the constraint $\clb$ defining the lower bound for assignments to \textsf{Classes} (Fig.~\ref{fig:CRA-constraint-lowerbound}) in Example~\ref{ex:constraintsGraphTGVsSubT}. 
    When we instantiate such a constraint from $\GraphTG$ in $\Sub{T}$, we will obtain a conjunction of constraints of which each one has the structure (compare Example~\ref{ex:instantiation} for this claim)
    \begin{equation*}
        c = \neg\exists \, \big(b_1\colon \emptyset \subseteq B_1, \textsf{true}\big) \vee \exists \, \big(b_1\colon \emptyset \subseteq B_1, \bigvee_{j=1}^t \exists \, (b_2^j\colon B_1 \subseteq B_2^j, \textsf{true})\big) .
    \end{equation*}
    Flattening such a constraint $c$ (over the identity morphism $id_{\emptyset}$) proceeds as follows (where, for shorter notation, we merely note down the morphisms and omit the graphs over which they are defined):
    \begin{align*}
        \Flat(id_{\emptyset}, c) & \stackrel{(F3)}{=} \Flat\big(id_{\emptyset}, \neg\exists \, (b_1, \textsf{true})\big) \vee \Flat\big(id_{\emptyset}, \exists \, (b_1, \bigvee_{j=1}^t \exists \, (b_2^j, \textsf{true}))\big)\\
         & \stackrel{(F6), (F5)}{=} \bigg(\exists \, (id_{\emptyset}, \textsf{true}) \wedge \neg \Flat\big(id_{\emptyset}, \exists \, (b_1, \textsf{true}) \big)\bigg) \vee \Flat\big(b_1 \circ id_{\emptyset}, \bigvee_{j=1}^t \exists \, (b_2^j, \textsf{true})\big)\\
         & \stackrel{(F5), (F3)}{=} \bigg(\exists \, (id_{\emptyset}, \textsf{true}) \wedge \neg \Flat\big(b_1 \circ id_{\emptyset}, \textsf{true} \big)\bigg) \vee \bigvee_{j=1}^t \Flat\big(b_1 \circ id_{\emptyset}, \exists \, (b_2^j, \textsf{true})\big)\\
         & \stackrel{(F1), (F5)}{=} \bigg(\exists \, (id_{\emptyset}, \textsf{true}) \wedge \neg \exists \, (b_1 \circ id_{\emptyset}, \textsf{true})\bigg) \vee \bigvee_{j=1}^t \Flat\big(b_2^j \circ b_1 \circ id_{\emptyset}, \textsf{true}\big)\\
         & \stackrel{(F1)}{=} \bigg(\exists \, (id_{\emptyset}, \textsf{true}) \wedge \neg \exists \, (b_1 \circ id_{\emptyset}, \textsf{true})\bigg) \vee \bigvee_{j=1}^t \exists \, (b_2^j \circ b_1 \circ id_{\emptyset}, \textsf{true})\\
         & \equiv \neg \exists \, (b_1, \textsf{true}) \vee \bigvee_{j=1}^t \exists \, (b_2^j \circ b_1, \textsf{true}) .
    \end{align*}
    This is exactly the structure of the constraint $\clbSubT{T}$ in Fig.~\ref{fig:CRA-constraint-lowerbound-SubT-unsimplified}. 
\end{example}

Next, though intuitively clear, we prove that applying the flattening construction computes a condition of nesting level $\leq1$.
In other words, it outputs a Boolean combination of literals. 

\begin{restatable}[Flattening results in non-nested condition]{lemma}{flatteningToLiterals}\label{lem:flattening-to-literals}
    Given a type graph $\TG$, a graph $T$ typed over it, a morphism $b_0$ with co-domain $B_0$ in $\Sub{T}$ and a condition $c$ over $B_0$, then 
    \begin{equation*}
        \nl(\Flat(b_0,c)) \leq 1 .
    \end{equation*}
    In particular, the condition resulting from flattening over a morphism is a Boolean combination of literals. 
\end{restatable}

Next, we show that flattening preserves the semantics of satisfaction in $\Sub{T}$; the proof proceeds via structural induction where the most interesting inductive step can be performed because of Lemma~\ref{lem:extracting-negation}. 

\begin{restatable}[Flattening preserves satisfaction in $\Sub{T}$]{theorem}{semanticsFlattening}\label{thm:semantics-flattening}
    Given a type graph $\TG$, a graph $T$ typed over it, a morphism $b_0\colon X \subseteq B_0$ in $\Sub{T}$ and a condition $c$ over $B_0$ in $\Sub{T}$, then, for any graph $G$ in $\Sub{T}$ and any inclusion $\incl{X}{G}\colon X \subseteq G$,
    \begin{equation*}
        \incl{X}{G} \Imodels \exists \, (b_0,c) \iff \incl{X}{G} \Imodels \Flat(b_0,c) .
    \end{equation*}
    In particular,
    \begin{equation*}
        \incl{B_0}{G} \Imodels c \iff \incl{B_0}{G} \Imodels \Flat(\mathit{id}_{B_0},c) . \qedhere
    \end{equation*}
\end{restatable}

The ability to flatten constraints in $\Sub{T}$ in a semantics-preserving way means we can assume all constraints and conditions to be given in some of the normal forms that are known from propositional logic. 
The following corollary states this---it just combines the results stated in Lemma~\ref{lem:flattening-to-literals} and Theorem~\ref{thm:semantics-flattening} with the well-known existence of normal forms for propositional logic (see, e.g., \cite[Proposition~2.6.1]{Avigad22}).
\begin{corollary}\label{cor:normal-forms-exist}
    Given a type graph $\TG$, a graph $T$ typed over it and a set $\mathcal{C}$ of conditions in $\Sub{T}$, there is a set $\mathcal{C}^{\mathrm{NF}}$ of conditions in $\Sub{T}$ such that (i) all conditions from $\mathcal{C}^{\mathrm{NF}}$ are Boolean combinations of literals over $\Sub{T}$ and are all in the same normal form known from propositional logic (e.g., \emph{conjunctive} or \emph{disjunctive normal form}) and (ii) there is a function $j\colon \mathcal{C} \to \mathcal{C}^{\mathrm{NF}}$ such that for every $c \in \mathcal{C}$ and every $S \in \Sub{T}$
    \begin{equation*}
        \incl{B_0}{S} \Imodels c \iff \incl{B_0}{S} \Imodels j(c) ,
    \end{equation*}
    where $B_0$ is the graph over which condition $c$ is defined and $\incl{B_0}{S}\colon B_0 \subseteq S$. 
\end{corollary}

For our work in~\cite{KZ26}, the \emph{conjunctive normal form} (CNF) provides a particularly convenient basis.
That is, every constraint or condition is a conjunction of disjunctions of literals over $\Sub{T}$. 
Considering the separate clauses (i.e., disjunctions of literals) as individual conditions further simplifies our work. 
The following semantic equivalence allows us to further restrict the kinds of conditions we need to consider: we can unite conjunctions of positive literals to a single literal. 
\begin{restatable}[Union of conjunctions]{lemma}{unionConjunctions}\label{lem:equivalence-union}
    Given a type graph $\TG$ and a graph $T$ typed over it, let $c_i = \exists \, (b_1^i\colon B_0 \subseteq B_1^i, \textsf{true})$, where $i = 1, \dots, k$, be a finite set of positive literals in $\Sub{T}$. 
    Then, for every $S \in \Sub{T}$, 
    \begin{equation*}
        \incl{B_0}{S} \Imodels \bigwedge_{i=1}^k c_i \iff \incl{B_0}{S} \Imodels \exists \, (b_1\colon B_0 \subseteq \bigcup_{i=1}^kB_1^i, \textsf{true})
    \end{equation*}
    where $\incl{B_0}{S}\colon B_0 \subseteq S$ is the inclusion of $B_0$ in $S$ and $b_1$ the inclusion of $B_0$ in $\bigcup_{i=1}^k B_1^i$. 
\end{restatable}

\paragraph{Presentation of Constraints.}
Wrapping up this section, in a lattice of subgraphs $\Sub{T}$, given a set $\mathcal{C}$ of conditions and constraints, we can assume each of it to be given in one of the following forms.
\begin{itemize}
    \item the \emph{trivial} cases: $c = \textsf{true}$ or $c = \textsf{false}$; 
    \item the \emph{purely positive} case: 
    \begin{equation*}
        c = \exists \, (B_1, \textsf{true}) \vee \dots \vee \exists \, (B_t, \textsf{true}); 
    \end{equation*}
    \item the \emph{purely negative} case:
    \begin{equation*}
        \neg\exists\, (D, \textsf{true}), 
    \end{equation*}
    which---by De Morgan's laws and  Lemma~\ref{lem:equivalence-union}---covers disjunctions of negative literals: 
    \begin{align*}
        c = \neg\exists \, (B_1, \textsf{true}) \vee \dots \vee \neg\exists \, (B_s, \textsf{true}) 
            & \equiv \neg (\exists \, (B_1, \textsf{true}) \wedge \dots \wedge \exists \, (B_s, \textsf{true}))\\
            & \Iequiv \neg\exists\, (D, \textsf{true}), \text{ where } D \coloneqq \bigcup_{i=1}^sB_i;
    \end{align*}
    \item the \emph{mixed} or \emph{implication} case:
    \begin{equation*}
        \exists\, (D, \textsf{true}) \implies \exists \, (B_1, \textsf{true}) \vee \dots \vee \exists \, (B_t, \textsf{true}) ,
    \end{equation*}
    which---again by De Morgan's laws and  Lemma~\ref{lem:equivalence-union}---covers clauses that contain both positive and negative literals: 
    \begin{align*}
        c & = \exists \, (B_1, \textsf{true}) \vee \dots \vee \exists \, (B_t, \textsf{true}) \vee \neg\exists \, (D_1, \textsf{true}) \vee \dots \vee \neg\exists \, (D_s, \textsf{true}) \\
        & \Iequiv \exists \, (B_1, \textsf{true}) \vee \dots \vee \exists \, (B_t, \textsf{true}) \vee \neg\exists\, (\bigcup_{i=1}^sD_i, \textsf{true}) \\
        & \equiv \exists\, (D, \textsf{true}) \implies \exists \, (B_1, \textsf{true}) \vee \dots \vee \exists \, (B_t, \textsf{true}) ,  \text{ where } D\coloneqq \bigcup_{i=1}^s D_i.
    \end{align*}
\end{itemize}

    \section{Instantiation of Conditions and Constraints}
    \label{sec:reduction-constraints-to-sub-t}
    Given a type graph $\TG$ and a finite graph $T$ typed over it, specifying formulas in $\Sub{T}$ is inefficient: rather than capturing the logical essence of the constraint, we have to enumerate every possible situation explicitly.
Instead, we would like to use the power of quantification and express constraints using the more common nested conditions in $\GraphTG$.
In this section, we show how to translate a nested condition from $\GraphTG$ into an equivalent nested condition in $\Sub{T}$. 
Practically, this increases the usability of nested conditions and constraints in $\Sub{T}$, and, formally, it demonstrates their expressiveness in that context: first-order properties can still be expressed in that logic, even though nesting collapses.\footnote{%
    Note that this is not longer true for lattices of subgraphs $\Sub{T}$ where $T$ is infinite. 
    In the CRA example, for instance, when given a graph $T$ containing infinitely many nodes of type \textsf{Method}, the constraint $\clb$ in $\GraphTG$ (see Fig.~\ref{fig:CRA-constraint-lowerbound}) still expresses that every \textsf{Method} is assigned to a \textsf{Class}. 
    To express its analogue $\clbSubT{T}$ in $\Sub{T}$, however, we would generally need the possibility to build an infinite conjunction (of disjunctions), explicitly requiring for every concrete \textsf{Method} to be assigned to one of the available \textsf{Classes}.%
} 

Our translation of conditions and constraints is based on \emph{instantiations}.
In an instantiation, a graph (or a morphism) from $\GraphTG$ is identified with an isomorphic subgraph (inclusion) in $\Sub{T}$.
Such instantiations are the central technical ingredient in the transformation of a nested constraint/condition in $\GraphTG$ into one in $\Sub{T}$.

\begin{definition}[Instantiation]
    Given a type graph $\TG$ and a graph $T$ typed over it, an \emph{instantiation $\mu = (j,\incl{B}{S})$ in $S\in\Sub{T}$} of a graph $C$ from $\GraphTG$ consists of an isomorphism $j\colon C \isomto B$, where $B \subseteq S \in \Sub{T}$ and $\incl{B}{S}\colon B \subseteq S$ denotes that inclusion; we sometimes also just call the image $B$ of an instantiation $\mu$ an instantiation.
    
    An \emph{instantiation in $S \in \Sub{T}$} of a morphism $a\colon C_0 \to C_1$ from $\GraphTG$ is a morphism $b\colon B_0 \subseteq B_1$ between the images of instantiations $\mu_0 = C_0 \isomto B_0 \subseteq S$ and $\mu_1 = C_1 \isomto B_1 \subseteq S$ of $C_0$ and $C_1$ in $S$, respectively, such that the diagram in Fig.~\ref{fig:instantiation-morphism} commutes.
    For instantiations in $T \in \Sub{T}$ we just speak of \emph{instantiations in $\Sub{T}$}.
    \begin{figure}%
        \centering
        \begin{tikzpicture}
            \matrix (m) [	matrix of math nodes,
                                        row sep=1.25em,
                                        column sep=1.25em,
                                        minimum width=1.25em,
                                        nodes in empty cells]
            {
                C_0	&							& C_1 \\
                    &							& \\
                B_0	&							& B_1 \\
                    & 							& \\
                    & S	                        & \\};
            \path[-stealth]
                (m-1-1) edge [->] node [above] {\scriptsize $a$} (m-1-3)
                                edge [draw=none] node [sloped] {$\simeq$} node [left,inner xsep=8pt] {\scriptsize $j_0$} (m-3-1)
                (m-1-3) edge [draw=none] node [sloped] {$\simeq$} node [right,inner xsep=8pt] {\scriptsize $j_1$} (m-3-3)
                (m-3-1) edge [draw=none] node {$\subseteq$} node [below, inner ysep=7pt] {\scriptsize $b$} (m-3-3)
                                edge [draw=none] node [sloped] {$\subseteq$} node [left,inner xsep=9pt,pos=.6] {\scriptsize $\incl{B_0}{S}$} (m-5-2)
                (m-3-3) edge [draw=none] node [sloped,allow upside down] {$\subseteq$} node [right,inner xsep=9pt,pos=.6] {\scriptsize $\incl{B_1}{S}$} (m-5-2);
        \end{tikzpicture}
        \caption{Instantiation of a morphism $a\colon C_0 \to C_1$ in $S$}%
        \label{fig:instantiation-morphism}%
    \end{figure}
\end{definition}

\begin{example}
    Consider the graphs $C_1$ and $C_2$ from the constraint $\clb$, depicted in Fig.~\ref{fig:CRA-constraint-lowerbound}, requiring every \textsf{Method} to be assigned to a \textsf{Class} in the CRA problem. 
    Instantiations of $C_1$ in the graph $S \subseteq T$ (see Fig.~\ref{fig:CRA-example-solution}) consist of a graph that contains one of the \textsf{Methods M1, M2} or \textsf{M3} as its single node, the isomorphism between $C_1$ and that graph, and the inclusion of that graph into $S$. 
    Similarly, instantiations of $C_2$ in $S$ consist of a graph that contains one of the \textsf{Classes C1} or \textsf{C2}, one of the \textsf{Methods M1, M2} or \textsf{M3} and an \textsf{encapsulates}-edge between them, of the isomorphism between $C_2$ and that graph, and the inclusion of that graph into $S$. 
    Given instantiations of $C_1$ and $C_2$ in $S$, the instantiation of the morphism $a_2\colon C_1 \hookrightarrow C_2$ is the inclusion between the instantiated graphs in $\Sub{T}$ (if this exists). 
    Instantiations in $T$ (instead of $S$) are similar, just that in $T$ more \textsf{Classes} are available to map to. 
\end{example}

As a rather obvious but important technical tool we note that instantiations correspond to injective morphisms; this will provide us the bridge between the ordinary notion of satisfaction of a condition in $\GraphTG$ and the one in $\Sub{T}$ (denoted as $\Imodels$).

\begin{restatable}[Correspondence of instantiations and injectives]{lemma}{correspondenceInstantiations}\label{lem:correspondence-instantiations}
    Given a type graph $\TG$ and a graph $T$ typed over it, for any graph $C$ from $\GraphTG$ and any graph $S \in \Sub{T}$ there exists a one-to-one correspondence between injective morphisms $q\colon C \hookrightarrow S$ in $\GraphTG$ and instantiations $\mu = (j, \incl{B}{S})$ of $C$ in $S$. 
    
    Furthermore, any instantiation $\mu = (j, \incl{B}{S})$ of a graph $C \in\GraphTG$ in a graph $S \in \Sub{T}$ can be extended to an instantiation $\mu^{T} = (j, \incl{B}{T})$ of $C$ in $\Sub{T}$. 
\end{restatable}

Furthermore, instantiations of morphisms are determined by the instantiations of their domain and codomain, and they are inclusions themselves.
\begin{restatable}[Uniqueness of instantiation]{lemma}{uniquenessInstantiation}\label{lem:unique-instantiation-morphism}
    Let $\TG$ be a type graph and $T$ be a graph typed over it.
    For every morphism $a\colon C_0 \to C_1$ in $\GraphTG$ and all pairs of instantiations $\mu_0 = (j_0, \incl{B_0}{S})$ of $C_0$ and $\mu_1 = (j_1, \incl{B_1}{S})$ of $C_1$ in some graph $S \in \Sub{T}$ there is an instantiation of the morphism $a$ if and only if 
    \begin{equation}\label{eq:existence-instantiating-morphism}
        \incl{B_1}{S} \circ j_1 \circ a = \incl{B_0}{S} \circ j_0. 
    \end{equation}
    In that case, moreover, the instantiating morphism $b\colon B_0 \subseteq B_1$ in $\Sub{T}$ is unique and an inclusion and the morphism $a$ is injective. 
\end{restatable}

In the following, we will often consider instantiations that stem from an injective morphism $q$. 
We will denote such an instantiation as $\mu_q = (\qe,\qm)$. 

The following theorem is the central result of this section, stating that conditions from $\GraphTG$ can equivalently be transformed into conditions from $\Sub{T}$ for finite graphs $T$ typed over $\TG$. 
The operation $\Inst$ appearing in its statement refers to the instantiation of constraints we develop in the following. 
Together with Theorem~\ref{thm:semantics-flattening} this means that, when working in $\Sub{T}$ one can really assume every first-order constraint to be given as discussed at the end of Sect.~\ref{sec:sub_t_constraints}. 

\begin{restatable}[Instantiation preserves semantics]{theorem}{instantiationPreservationSemantics}\label{thm:instantiation-preserves-semantics}
    Given a type graph $\TG$ and a \emph{finite} graph $T$ typed over it, let $c$ be a nested condition in $\GraphTG$ over some graph $C_0$. 
    Then, for every graph $S \in \Sub{T}$ and every injective morphism $g\colon C_0 \hookrightarrow S$,
    \begin{equation*}
        g \models c \iff \gm \Imodels \Inst(c,\mu_g^{T},\Sub{T}) 
    \end{equation*}
    where $\mu_g = (\gep,\gm)$ is the instantiation of $C_0$ in $S$ that corresponds to $g$ and $\mu_g^{T}$ is the corresponding instantiation in $T$ (both guaranteed to uniquely exist by Lemma~\ref{lem:correspondence-instantiations}).
\end{restatable}

Next, we develop the instantiation process $\Inst$ referred to in Theorem~\ref{thm:instantiation-preserves-semantics}.
This process relies on two intuitive ideas: 
First, an existentially quantified graph $C$ is replaced by a disjunction over all ways in which it can be instantiated in $\Sub{T}$. 
Second, a subcondition $\neg \exists \, (C, d)$ states that at all occurrences of $C$ (that are compatible with the considered instantiation of the graph $C_0$ over which the condition is defined) the condition $d$ must \emph{not} hold. 
We therefore replace such subconditions by a conjunction which, for every possible instantiation of $C$, states that the occurrence of such an instantiation implies the condition $\neg d$. 
The formal treatment of this second idea is somewhat technical as we need to account for the fact that no occurrence of $C$ might exist and need to take care to still obtain a nested condition over the correct graph. 
We clarify these details in Lemma~\ref{lem:negation-instantiation}, after we have developed and illustrated our instantiation process.  

\begin{construction}[Instantiating conditions and constraints]
    Given a type graph $\TG$ and a finite graph $T$ typed over it, let $c$ be a nested condition over $C_0$ in $\GraphTG$.  
    Given an instantiation $\mu_0 = (j\colon C_0 \isomto B_0, \incl{B_0}{T} \colon B_0 \subseteq T)$ of $C_0$ in $\Sub{T}$, the \emph{instantiated variant of $c$} (in $\Sub{T}$ with respect to $\mu_0$), denoted as $\Inst(c,\mu_0,\Sub{T})$, is recursively constructed as follows:
    \begin{itemize}
        \item for $c = \textsf{true}$, $\Inst\big(c,\mu_0,\Sub{T}\big) = \textsf{true}$;

        \item for $c = d_1 \odot d_2$, where $\odot \in \{\wedge, \vee\}$,
        \begin{equation*}
            \Inst\big(c,\mu_0,\Sub{T}\big) = \Inst\big(d_1,\mu_0,\Sub{T}\big) \odot \Inst\big(d_2,\mu_0,\Sub{T}\big);
        \end{equation*}
        
        \item for $c = \exists \, (a_1 \colon C_0 \hookrightarrow C_1,d)$, where $d$ is a condition over $C_1$, 
        \begin{equation*}
            \Inst\big(c,\mu_0,\Sub{T}\big) = \bigvee_{b_1 \in \Jcal} \exists \, \Big(b_1\colon B_0 \subseteq B_1, \Inst\big(d,\mu_1,\Sub{T}\big)\Big),
        \end{equation*}
        where $\mathcal{J}$ indexes instantiations of $a_1$ with domain $\mu_0$ in $\Sub{T}$ (and $\mu_1$ denotes the respective codomain); the empty disjunction is interpreted as $\textsf{false} = \neg \textsf{true}$;
        
        \item for $c = \neg d$, we make use of a look ahead of size 1 and set $\Inst\big(c,\mu_0,\Sub{T}\big) =$
        \begin{equation*}
            \left\{\begin{aligned}
                \textsf{false},																																& \text{ for } d = \textsf{true} \\
                \Inst\big(\neg e_1,\mu_0,\Sub{T}\big) \overline{\odot} \Inst\big(\neg e_2,\mu_0,\Sub{T}\big),		& \text{ for } d = e_1 \odot e_2 \\
                \Inst\big(e,\mu_0,\Sub{T}\big),																									& \text{ for } d = \neg e \\
                \bigwedge_{b_1 \in \Jcal} \Big( \neg\exists\, (b_1\colon B_0 \subseteq B_1, \textsf{true}) \vee & \\
                \exists \, \big(b_1\colon B_0 \subseteq B_1, 
                \Inst(\neg e,\mu_1,\Sub{T})\big)\Big), & \text{ for } d = \exists \, (a_1 \colon C_0 \hookrightarrow C_1,e) ,
            \end{aligned}\right.
        \end{equation*}
        where, again, $\mathcal{J}$ indexes instantiations of $a_1$ with domain $\mu_0$, $\mu_1$ denotes the respective codomain, the empty conjunction is interpreted as \textsf{true}, and $\odot \in \{\wedge, \vee\}$, $\overline{\wedge} = \vee$, and $\overline{\vee} = \wedge$.
    \end{itemize}
    
    Given a constraint $c$, the \emph{instantiated variant of $c$ in $\Sub{T}$}, denoted as $\Inst(c,\Sub{T})$, is defined as
    \begin{equation*}
        \Inst(c,\Sub{T}) = \Inst(c,\mu_0,\Sub{T}) ,
    \end{equation*}
    where $\mu_0$ is the unique instantiation of $\emptyset$ in $T$. 
\end{construction}

\begin{example}\label{ex:instantiation}
    A kind of constraint one regularly encounters in practice and that is tedious to define in $\Sub{T}$ are constraints of the structure $\forall \, (C_1, \exists \, C_2)$. 
    In our running example of the CRA problem, this concerns the constraint $\clb$, requiring every \textsf{Method} and \textsf{Attribute} to be assigned to at least one \textsf{Class} (see Fig.~\ref{fig:CRA-constraint-lowerbound} in Example~\ref{ex:constraintsGraphTGVsSubT}).
    Instantiating such a constraint amounts to the following: 
    First, $\forall \, (C_1, \exists \, C_2)$ is syntactic sugar for $\neg \exists \, (a_1\colon \emptyset \hookrightarrow C_1, \neg \exists\, (a_2\colon C_1 \hookrightarrow C_2, \textsf{true}))$. 
    Thus, by construction, we obtain\begin{equation*}
        \begin{split}
            \Inst(c_{\mathrm{lb}}, \Sub{T}) =	& \bigwedge_{b_1 \in \Jcal} \bigg(\neg\exists \, (b_1\colon \emptyset \subseteq B_1, \textsf{true}) \vee \\
                                                                        & \exists \, \Big(b_1\colon \emptyset \subseteq B_1, \Inst\big(\neg ( \neg\exists \, (a_2\colon C_1 \hookrightarrow C_2, \textsf{true})), \mu_1,\Sub{T}\big)\Big)\bigg) ,
        \end{split}
    \end{equation*}
    where $\mathcal{J}$ indexes morphisms $b_1$ from $\emptyset$ to instantiations of $C_1$ in $T$. 
    Continuing the computation, we receive
    \begin{align*}
            & \Inst(\neg (\neg\exists \, (a_2 \colon C_1 \hookrightarrow C_2, \textsf{true}), \mu_1,\Sub{T}) \\
        =   & \Inst(\exists \, (a_2 \colon C_1 \hookrightarrow C_2, \textsf{true}), \mu_1,\Sub{T}) \\
        =   & \bigvee_{b_2 \in \mathcal{J}_{b_1}} \Big( \exists\, \big(b_2\colon B_1 \subseteq B_2, \Inst(\textsf{true}, \mu_2, \Sub{T}) \big) \Big) \\
        =   & \bigvee_{b_2 \in \mathcal{J}_{b_1}} \exists\, \big(b_2\colon B_1 \subseteq B_2, \textsf{true} \big), 
    \end{align*}
    where $\mathcal{J}_{b_1}$ indexes those instantiations of the morphism $a_2$ whose domain is instantiated to the co-domain of $b_1$. 
    Thus, in total, we receive a conjunction of constraints of the form 
    \begin{align*}
        & \neg\exists \, (b_1\colon \emptyset \subseteq B_1, \textsf{true}) \vee \exists \, \Big(b_1\colon \emptyset \subseteq B_1, \bigvee_{b_2 \in \mathcal{J}_{b_1}} \exists\, \big(b_2\colon B_1 \subseteq B_2, \textsf{true} \big) \Big) \\
       = & \neg\exists \, (b_1\colon \emptyset \subseteq B_1, \textsf{true}) \vee \bigvee_{b_2 \in \mathcal{J}_{b_1}} \exists \, \big(b_2 \circ b_1\colon \emptyset \subseteq B_2, \textsf{true} \big),
    \end{align*}
    where $B_1$ ranges over the possibilities to instantiate $C_1$ in $T$ and $B_2$ over those to instantiate $C_2$ in a way conformant to $B_1$, the instantiation chosen for $C_1$. 
   
    Given $T$ as introduced in Fig.~\ref{fig:CRA-example-solution-T} in Example~\ref{ex:container-and-subgraph}, for $\clb$ the graph $C_1$ can be instantiated in three different ways, mapping its single node of type \textsf{Method} to one of the \textsf{Methods M1, M2, M3} of $T$. 
    For each such instantiation of $C_1$, $C_2$ (and the morphism from $C_1$ to $C_2$) can be instantiated compatibly by mapping its \textsf{Method} in the same way, its \textsf{Class} to one of the \textsf{Classes C1, \dots, C6} of $T$, and the \textsf{encapsulates}-edge accordingly. 
    Figure~\ref{fig:CRA-constraint-lowerbound-SubT-unsimplified} shows one of the conjuncts of the resulting constraint $\clbSubT{T}$, namely the one mapping $C_1$'s node to \textsf{Method M1}.
    
\end{example}

Like Theorem~\ref{thm:semantics-flattening}, also Theorem~\ref{thm:instantiation-preserves-semantics} is proved by structural induction. 
The following lemma, to which we already pointed earlier, is the key technical ingredient in the induction step, clarifying that our instantiation of negated conditions preserves the semantics of conditions. 

\begin{restatable}[Negation in instantiation]{lemma}{negationInstantiation}\label{lem:negation-instantiation}
    Given a type graph $\TG$ and a finite graph $T$ typed over it, let $c = \neg d$ be a nested condition over $C_0$ that is also typed over $\TG$. 
    Then, for any instantiation $\mu_0 = (j\colon C_0 \isomto B_0, \incl{B_0}{T} \colon B_0 \subseteq T)$ of $C_0$ in $\Sub{T}$ and any inclusion $\incl{B_0}{S}\colon B_0 \subseteq S$, it holds that 
    \begin{equation*}
        \incl{B_0}{S} \Imodels \Inst\big(d,\mu_0, \Sub{T}\big) \iff \incl{B_0}{S} \nImodels \Inst\big(\neg d,\mu_0, \Sub{T}\big) . \qedhere
    \end{equation*}
\end{restatable}

Finally, we shortly clarify that an instantiated condition still consists of finitely many graphs; in particular, this ensures that the conjunction/disjunction indexed by a set $\Jcal$ of instantiations is well-defined. 
It also clarifies the severity of the growth of the number of graphs occurring in a condition when instantiating it.
\begin{restatable}[Order of number of instantiated constraints/conditions]{proposition}{orderInstantiation}\label{prop:order-instantiation}
    Given a type graph $\TG$ and a finite graph $T$ that is typed over it, let $c$ be some condition over a graph $C_0$ that is typed over $\TG$, and let $\lvert c\rvert$ denote the number of morphisms occurring in $c$. 
    Then, for every instantiation $\mu_0$ of $C_0$ in $\Sub{T}$,     
    \begin{equation*}
        \lvert \Inst(c, \mu_0, \Sub{T})\rvert \in \mathcal{O}(\lvert c\rvert \cdot \lvert T\rvert^{\lvert \Cmax\rvert}),
    \end{equation*}
    where $\Cmax$ is the largest graph that occurs in $c$ and $\lvert X\rvert$ denotes the number of elements of a graph $X$.
    In particular, all conjunctions and disjunctions in $\Inst(c, \mu_0, \Sub{T})$ are finite, i.e., it is a nested condition. 
    Moreover, the total number of instantiated variants of $c$ in $\Sub{T}$ is also in $\mathcal{O}(\lvert c\rvert \cdot \lvert T\rvert^{\lvert \Cmax\rvert})$. 
\end{restatable}
While the size of $\Inst(c, \mu_0, \Sub{T})$, thus, grows exponentially in the size of the largest graph occurring in $c$, this tends to be not a problem in practice as graphs in nested graph conditions are usually small.

    \section{Conclusion}
    \label{sec:conclusion}
    In this paper, we have presented a semantics-preserving translation of nested conditions and constraints from the category $\GraphTG$ (of graphs typed over a given type graph $\TG$) into nested conditions and constraints formulated in a category $\Sub{T}$, the lattice of subgraphs of a fixed finite graph $T$ typed over $\TG$ (Theorem~\ref{thm:instantiation-preserves-semantics}). 
In $\Sub{T}$, moreover, every nested condition and constraint can be equivalently transformed into a nesting-free normal form (Theorem~\ref{thm:semantics-flattening} and Corollary~\ref{cor:normal-forms-exist}). 
Together, both results show that in $\Sub{T}$ first-order properties can be expressed by the nesting-free fragment of the formalism of nested conditions and constraints (given finiteness of $T$). 

These results are not surprising---we work in a fixed, finite universe and can replace nesting by explicitly enumerating all concrete situations in a formula. 
Our theoretical contribution, therefore, consists in making the expected precise, also clarifying details that are slightly involved (the treatment of negation). 
Practically, however, we see potential for interesting applications of our results: 
A finite lattice of subgraphs $\Sub{T}$ is a quite natural setting for many problems in model-driven optimisation~\cite{JKLT23} or model repair~\cite{MacedoTC17}, and with our work we provide the formal bridge for conveniently specifying constraints as general nested conditions (in $\GraphTG$) but using the far more concrete constraints (in $\Sub{T}$) for downstream tasks. 
In fact, we have developed the theory in this paper to enable such an application, namely the derivation of non-blocking consistency-preserving rules in a finite lattice of subgraphs $\Sub{T}$~\cite{KZ26}.

There are two immediate avenues for generalising our theoretical results:
First, in our proofs we do not make use of the fact that our objects are graphs; what we critically use, though (in Sect.~\ref{sec:reduction-constraints-to-sub-t}), is the fact that the graph $T$, whose subgraphs form our category of interest, is finite and, therefore, there is only a finite number of injective morphisms from a graph $C$ to a subgraph of $T$. 
This suggests that our results should directly transfer from graphs to the far more general setting of categories of subobjects of a container object that stems from any \emph{finitary} $\mathcal{M}$-adhesive category~\cite{GabrielBEG14}. 
Moreover, in our setting of a finite lattice of subgraphs, it is known that first- and second-order logics coincide (see~\cite{ElberfeldGT16}, in particular also for far more interesting graph classes where first-order logic coincides with some fragment of second-order logic). 
This means that it is possible to develop translations, similar to the one presented here, also for formalisms that have been developed to extend nested conditions and constraints beyond first-order as, e.g.,~\cite{Radke12, NavarroOPL21, DrewesHM25}.

    \begin{acknowledgments}
        The authors thank Alexander Lauer for helpful comments on a first draft of this paper. 
        This work was partially funded by the German Research Foundation (DFG), project \enquote{Model-Driven Optimization in Software Engineering} (\href{https://gepris.dfg.de/gepris/projekt/462887453?language=en}{TA 294/19-1}). 
    \end{acknowledgments}

    \section*{Declaration on Generative AI}
    The authors have not employed any Generative AI tools.

	\bibliography{literature.bib}

    \iflong
        \appendix
    
        \section{Proofs}
        \label{app:proofs}
        In this appendix, we present all proofs of our results. 

\extractingNegation*

\begin{proof}
    Let $c = \exists \, (b_0\colon B_0 \subseteq B_1, \neg d)$ be such a condition in $\Sub{T}$ with $d$ being a condition over $B_1 \subseteq T$ and let $c' = \exists \, (b_0 \colon B_0 \subseteq B_1, \textsf{true}) \wedge \neg \exists \, (b_0 \colon B_0 \subseteq B_1, d)$. 
    Since there is maximally one morphism between any two elements of $\Sub{T}$, namely an inclusion, the stated equivalence, $c \Iequiv c'$, holds if for every $\incl{B_0}{G}\colon B_0 \subseteq G$, where $B_0 \subseteq G \subseteq T$ and $\incl{B_0}{G}$ denotes the inclusion, $\incl{B_0}{G} \Imodels c$ if and only if $\incl{B_0}{G} \Imodels c'$.
    By definition of the semantics of conditions, such an inclusion $\incl{B_0}{G}$ satisfies $c$ if and only if there exists an inclusion $\incl{B_1}{G}\colon B_1 \subseteq G$ such that $\incl{B_1}{G} \nvDash d$. 

    Similarly, also by definition of the semantics, $\incl{B_0}{G}$ satisfies $c'$ if it satisfies $\exists \, (b_0 \colon B_0 \subseteq B_1, \textsf{true})$ and also satisfies $\neg \exists \, (b_0 \colon B_0 \subseteq B_1, d)$. 
    This is the case if and only if $B_1 \subseteq G$, i.e., $\incl{B_1}{G}\colon B_1 \subseteq G$ exists, but simultaneously $\incl{B_0}{G}$ does \emph{not} satisfy $\exists \, (b_0 \colon B_0 \subseteq B_1, d)$. 
    By existence of $\incl{B_1}{G}$ this can only be the case if $\incl{B_1}{G} \nvDash d$. 
    Therefore, $c \Iequiv c'$. 
\end{proof}

\flatteningToLiterals*

\begin{proof}
    By the recursive definition of conditions, the recursive flattening construction always ends with applying the rules $\Flat(b_0, \textsf{true}) = \exists \, (b_0, \textsf{true})$ or $\Flat(b_0, \neg\textsf{true}) = \neg\textsf{true}$, resulting in a (negated) literal of nesting depth $\leq 1$. 
    Among the further rules applied during the construction, only the rule $\Flat(b_0, \neg \exists \, (b_1\colon B_0 \subseteq B_1, d) = \exists \, (b_0, \textsf{true}) \wedge \neg \Flat(b_1 \circ b_0, d)$ directly outputs a condition, namely $\exists \, (b_0, \textsf{true})$, and $\nl(\exists \, (b_0, \textsf{true})) = 1$. 
    All other rules just determine the structure of the Boolean combination of the literals computed as output, but Boolean combinations of literals do not increase nesting depth.
\end{proof}

\semanticsFlattening*

\begin{proof}
    For conditions, it generally holds that 
    \begin{equation}\label{eq:composition-existential-quantification}
        \exists\, (b_1, \exists\, (b_2, d)) \equiv \exists\, (b_2 \circ b_1, d) ;
    \end{equation} 
    and even $ \exists\, (j, d) \equiv d$ for any isomorphism $j$~\cite[Fact~4]{HP05}. 
    With that, the second claim of the theorem immediately follows from the first: 
    By the first claim we have $\incl{B_0}{G} \Imodels \Flat(\mathit{id}_{B_0},c) \equiv \incl{B_0}{G} \Imodels \exists \, (\mathit{id}_{B_0},c)$, which is always equivalent to $\incl{B_0}{G} \Imodels c$. 
    
    We prove the first claim via structural induction, where the first two items in the definition of flattening constitute the bases for this induction. 
    For these base cases, $\Flat(b_0,c)$ even equals $\exists \, (b_0,c)$ by definition. 

    Assuming the statement of the first claim to be true for some condition $d$ (IH), the induction step is routine for the cases merely involving Boolean connectives. 
    For the two cases involving existential quantifiers we compute
    \begin{align*}
        \incl{X}{G} \Imodels \Flat(b_0, \exists\, (b_1, d)) & \overset{\text{Def}}{\iff} \incl{X}{G} \Imodels \Flat(b_1 \circ b_0, d) \\
                                                            & \overset{\text{(IH)}}{\iff} \incl{X}{G} \Imodels \exists \, (b_1 \circ b_0, d)\\
                                                            & \overset{\text{Eq.}~\ref{eq:composition-existential-quantification}}{\iff} \incl{X}{G} \Imodels \exists \, (b_0, \exists\, (b_1, d))
    \end{align*}
    and 
    \begin{align*}
        \incl{X}{G} \Imodels \Flat(b_0, \neg \exists\, (b_1, d))    & \overset{\text{Def}}{\iff} \incl{X}{G} \Imodels \exists\, (b_0, \textsf{true}) \wedge \neg \Flat(b_1 \circ b_0, d) \\
                                                                    & \overset{\text{(IH)}}{\iff} \incl{X}{G} \Imodels \exists\, (b_0, \textsf{true}) \wedge \neg\exists \, (b_1 \circ b_0, d)\\
                                                                    & \overset{\text{Eq.}~\ref{eq:composition-existential-quantification}}{\iff} \incl{X}{G} \Imodels \exists\, (b_0, \textsf{true}) \wedge \neg \exists \, (b_0, \exists\, (b_1, d))\\
                                                                    & \overset{\text{Lem.~}\ref{lem:extracting-negation}}{\iff} \incl{X}{G} \Imodels \exists\, (b_0, \neg \exists\,(b_1,d)). \qedhere
    \end{align*} 
\end{proof}

\unionConjunctions*

\begin{proof}
    There exists an inclusion for every $B_1^i$ in $S$ if and only if there exists an inclusion of $\bigcup_{i=1}^kB_1^i$ in $S$.
\end{proof}

\correspondenceInstantiations*

\begin{proof}
    For any instantiation $\mu = (j, \incl{B}{S})$, the morphism $\incl{B}{S} \circ j \colon C \to S$ is injective because $ \incl{B}{S}$ and $j$ are. 
    In the other direction, for every injective morphism $q\colon C \hookrightarrow S$, there exists the unique image-factorisation $q = \qm \circ \qe$ with $\qe\colon C \to \Ima q$ and $\qm\colon \Ima q \to S$, where we fix $\Ima q$ to be a subgraph of $S$ (and not merely an isomorphic copy of the image). 
    By construction, $\qe$ is surjective and $\qm$ an inclusion; by injectivity of $q$ and $\qm$, $\qe$ is additionally injective and therefore an isomorphism. 
    Both mappings (composition and image-factorisation) are inverses to each other. 
    Hence, the claimed one-to-one correspondence is established.
    
    Given an instantiation $\mu = (j, \incl{B}{S})$, the instantiation $\mu^{T} = (j, \incl{B}{T})$ in $\Sub{T}$ is obtained by just composing $\incl{B}{S}$ with the inclusion $ \incl{S}{T}\colon S \subseteq T$; i.e., $\incl{B}{T} \coloneqq  \incl{S}{T} \circ \incl{B}{S}$.
\end{proof}

\uniquenessInstantiation*

\begin{proof}
    For the following, compare Fig.~\ref{fig:instantiation-morphism}. 
    By invertibility of $j_0$, setting $b \coloneqq j_1 \circ a \circ j_0^{-1}$ is the only possible definition of $b$ that makes the upper square commute. 
    If $\incl{B_1}{S} \circ j_1 \circ a \neq \incl{B_0}{S} \circ j_0$, by definition there cannot exist an instantiation of $a$ as morphism from $\mu_0$ to $\mu_1$.  
    If this equation holds, however, we compute
    \begin{align*}
        \incl{B_1}{S} \circ b \circ j_0 & = \incl{B_1}{S} \circ j_1 \circ a \circ j_0^{-1} \circ j_0 \\
                                                    & = \incl{B_1}{S} \circ j_1 \circ a \\
                                                    & = \incl{B_0}{S} \circ j_0, 
    \end{align*}
    which implies $\incl{B_1}{S} \circ b = \incl{B_0}{S}$ by surjectivity of $j_0$. 
    
    Furthermore, as every instantiated morphism $b$ satisfies $\incl{B_1}{S} \circ b = \incl{B_0}{S}$, and $\incl{B_0}{S},\incl{B_1}{S}$ are both injective, also $b$ is injective which makes it an embedding (as both $B_0$ and $B_1$ are subgraphs of $S$). 
    Finally, every morphism $a$ that satisfies Eq.~\ref{eq:existence-instantiating-morphism} is injective because $\incl{B_0}{S} \circ j_0$ and $\incl{B_1}{S} \circ j_1$ are.
\end{proof}

\negationInstantiation*

\begin{proof}
    We prove the statement via structural induction.
    With regard to the inductive basis, for $d = \textsf{true}$, clearly
    \begin{equation*}
        \incl{B_0}{S} \nImodels \Inst\big(\neg d,\mu_0, \Sub{T}\big) = \textsf{false} \iff \incl{B_0}{S} \Imodels \Inst\big(d,\mu_0, \Sub{T}\big) = \textsf{true} .
    \end{equation*}
    
    Assuming the statement to hold for conditions $e_1$ and $e_2$ (induction hypothesis), for $c = \neg d$ with $d= e_1 \wedge e_2$ we can compute 
    \begin{align*}
                    & \incl{B_0}{S} \nImodels \Inst\big(\neg (e_1 \wedge e_2),\mu_0, \Sub{T}\big) \\	
        \iff	& \incl{B_0}{S} \nImodels \Big(\Inst\big(\neg e_1,\mu_0, \Sub{T}\big) \vee \Inst\big(\neg e_1,\mu_0, \Sub{T}\big)\Big) \\
        \iff	& \incl{B_0}{S} \nImodels \Inst\big(\neg e_1,\mu_0, \Sub{T}\big) \text{ and } \incl{B_0}{S} \nImodels \Inst\big(\neg e_1,\mu_0, \Sub{T}\big) \\
        \iff	& \incl{B_0}{S} \Imodels \Inst\big(e_1,\mu_0, \Sub{T}\big) \text{ and } \incl{B_0}{S} \Imodels \Inst\big(e_1,\mu_0, \Sub{T}\big) \\
        \iff	& \incl{B_0}{S} \Imodels \Inst\big(e_1 \wedge e_2,\mu_0, \Sub{T}\big) ;
    \end{align*}
    an analogous computation shows the case for $d = e_1 \vee e_2$. 
    
    Assuming the statement to hold for a condition $e$, for $c = \neg d$ with $d = \neg e$ we compute
    \begin{align*}
        \incl{B_0}{S} \nImodels \Inst\big(\neg \neg e,\mu_0, \Sub{T}\big)	& \iff \incl{B_0}{S} \nImodels \Inst\big(e,\mu_0, \Sub{T}\big) \\
                                                                                                                    & \iff \incl{B_0}{S} \models \Inst\big(\neg e,\mu_0, \Sub{T}\big) .
    \end{align*}
    
    And finally, assuming the statement to hold for a condition $e$ (over a graph $C_1$), for $c = \neg d$ with $d = \exists \, (a_1\colon C_0 \hookrightarrow C_1, e)$ we compute
    \begin{align*}
                    & \incl{B_0}{S} \nImodels \Inst\big(\neg \exists \, (a_1\colon C_0 \hookrightarrow C_1, e), \mu_0, \Sub{T} \big)\\
        \iff	& \incl{B_0}{S} \nImodels \bigwedge_{b_1 \in \Jcal} \neg\exists\, (b_1\colon B_0 \subseteq B_1, \textsf{true}) \vee \\
                    & \exists \, \big(b_1\colon B_0 \subseteq B_1, \Inst(\neg e,\mu_1,\Sub{T})\big) \\
        \iff	& \text{there exists } b_1 \in \Jcal \text{ s.t. } \incl{B_0}{S} \nImodels \neg\exists\, (b_1\colon B_0 \subseteq B_1, \textsf{true}) \vee \\
                    & \exists \, \big(b_1\colon B_0 \subseteq B_1, \Inst(\neg e,\mu_1,\Sub{T})\big) \\
        \iff	& \text{there exists } b_1 \in \Jcal \text{ s.t. there exists an inclusion } \incl{B_1}{S}\colon B_1 \subseteq S \text{ s.t.}\\
                    & \incl{B_1}{S} \circ b_1 = \incl{B_0}{S} \text{ and, for every such inclusion, } \incl{B_1}{S} \nImodels \Inst\big(\neg e, \mu_1, \Sub{T}\big) \\
        \iff	& \text{there exists } b_1 \in \Jcal \text{ s.t. there exists an inclusion } \incl{B_1}{S}\colon B_1 \subseteq S \text{ s.t.}\\
                    & \incl{B_1}{S} \circ b_1 = \incl{B_0}{S} \text{ and, for this inclusion, } \incl{B_1}{S} \nImodels \Inst\big(\neg e, \mu_1, \Sub{T}\big) \\
        \iff	& \text{there exists } b_1 \in \Jcal \text{ s.t. there exists an inclusion } \incl{B_1}{S}\colon B_1 \subseteq S \text{ s.t.}\\
                    & \incl{B_1}{S} \circ b_1 = \incl{B_0}{S} \text{ and, for this inclusion, } \incl{B_1}{S} \Imodels \Inst\big(e, \mu_1, \Sub{T}\big) \\
        \iff	& \incl{B_0}{S} \Imodels \Inst\big(\exists \, (a_1\colon C_0 \hookrightarrow C_1, e), \mu_0, \Sub{T} \big) .
    \end{align*}
    
    In this last computation, for the fourth equivalence we use the fact that there can only be one inclusion of a graph $B_1$ in a graph $S$---that is, only one morphism from $B_1$ to $S$ in $\Sub{T}$. 
\end{proof}

\instantiationPreservationSemantics*

\begin{proof}[Proof of Theorem~\ref{thm:instantiation-preserves-semantics}]
    We prove the claim via structural induction. 
    The assertion trivially holds for \textsf{true} (induction basis).
    Let $d$ be a condition over $C_1$ in $\GraphTG$ such that 
    \begin{equation*}
        q \models d \iff \qm \Imodels \Inst(d,\mu_q^{T},\Sub{T}) 
    \end{equation*}
    for every injective morphism $q\colon C_1 \hookrightarrow S$ for every $S \in \Sub{T}$ (induction hypothesis). 
    The inductive steps for $\wedge$ and $\vee$ are trivial. 
    For $c = \neg d$, we can use Lemma~\ref{lem:negation-instantiation} to obtain
    \begin{align*}
        q \models \neg d	& \iff q \nvDash d \\
                                            & \iff \qm \nImodels \Inst(d,\mu_q^{T},\Sub{T}) \\
                                            & \iff \qm \Imodels \Inst(\neg d,\mu_q^{T},\Sub{T}) .
    \end{align*}
    
    Lastly, we show the inductive step for the existential quantifier. 
    For every injective morphism $g \colon C_0 \hookrightarrow S$ where $S \in \Sub{T}$, by definition, $g \models c$ if and only if there exists an injective morphism $q\colon C_1 \hookrightarrow S$ such that $q \circ a_1 = g$ and $q \models d$. 
    By Lemma~\ref{lem:unique-instantiation-morphism}, there then exists a unique instantiation $b_1\colon B_0 \subseteq B_1$ between the instantiation of $C_0$ that is provided by $g$ and the one of $C_1$ that is provided by $q$, and all instantiations extend to instantiations in $T$; Fig.~\ref{fig:instantiation-morphism-proof} depicts the situation. 
    In particular, since for $\gmSPI\colon B_0 \subseteq T \coloneqq \incl{S}{T} \circ \gm$ and $\qmSPI\colon B_1 \subseteq T \coloneqq \incl{S}{T} \circ \qm$ we have $\qmSPI \circ b_1 = \gmSPI$, 
    \begin{equation*}
        \exists \, \big(b_1 \colon B_0 \subseteq B_1, \Inst(d,\mu_q^{T},\Sub{T})\big)
    \end{equation*}
    is one of the terms that occur in the disjunction $\Inst(c,\mu_g^{T},\Sub{T})$. 
    Summarizing, $\qm \circ b_1 = \gm$ and $\qm \Imodels \Inst(d,\mu_q^{T},\Sub{T})$ imply 
    \begin{equation*}
        \gm \Imodels \exists \, \big(b_1 \colon B_0 \subseteq B_1, \Inst(d,\mu_q^{T},\Sub{T})\big)
    \end{equation*}
    which implies $\gm \Imodels \Inst(c,\mu_g^{T},\Sub{T})$.
    
    \begin{figure}%
        \centering
        \begin{tikzpicture}
            \matrix (m) [	matrix of math nodes,
                                        row sep=1.25em,
                                        column sep=1.25em,
                                        minimum width=1.25em,
                                        nodes in empty cells]
            {
                C_0	&										& C_1 \\
                                                &										& \\
                B_0	&										& B_1 \\
                                                & 									& \\
                                                & S			& \\
                                                &										& \\
                                                & T	& \\};
            \path[-stealth]
                (m-1-1) edge [->] node [above] {\scriptsize $a_1$} (m-1-3)
                                edge [draw=none] node [sloped] {$\simeq$} node [left,inner xsep=7pt] {\scriptsize $\gep$} (m-3-1)
                                edge [left hook->,bend right=60] node [left] {\scriptsize $g$} (m-5-2)
                (m-1-3) edge [draw=none] node [sloped] {$\simeq$} node [right,inner xsep=7pt] {\scriptsize $\qe$} (m-3-3)
                                edge [right hook->,bend left=60] node [right] {\scriptsize $q$} (m-5-2)
                (m-3-1) edge [->,dashed] node [below] {\scriptsize $b_1$} (m-3-3)
                                edge [draw=none] node [sloped] {$\subseteq$} node [left,inner xsep=8pt,pos=.6] {\scriptsize $\gm$} (m-5-2)
                (m-3-3) edge [draw=none] node [sloped,allow upside down] {$\subseteq$} node [right,inner xsep=8pt,pos=.6] {\scriptsize $\qm$} (m-5-2)
                (m-5-2) edge [draw=none] node [sloped] {$\subseteq$} node [right,inner xsep=8pt] {\scriptsize $\incl{S}{T}$} (m-7-2);
        \end{tikzpicture}
        \caption{Instantiation of the morphism $a\colon C_0 \hookrightarrow C_1$ for the instantiations $\mu_g$ and $\mu_q$}%
        \label{fig:instantiation-morphism-proof}%
    \end{figure}
    
    If, conversely, there exists an instantiation $\mu_g = (\gep, \gm)$ such that $\gm \Imodels \Inst(c,\mu_g^{T},\Sub{T})$, there exists (at least) one term 
    \begin{equation*}
        \exists \, \big(b_1 \colon B_0 \subseteq B_1, \Inst(d,\mu_1,\Sub{T})\big)
    \end{equation*}
    in the disjunction $\Inst(c,\mu_g^{T},\Sub{T})$ that $\gm$ satisfies (in $\Sub{T}$). 
    In particular, there exists an inclusion $\qm\colon B_1 \subseteq S$ such that 
    \begin{equation*}
        \qm \Imodels \Inst(d,\mu_1,\Sub{T}) \text{ and } \qm \circ b_1 = \gm.
    \end{equation*}
    Since inclusions are unique, it holds that $\incl{S}{T} \circ \qm = \incl{C_1}{T}$, where $\mu_1 = (j_1, \incl{C_1}{T})$ is the instantiation of $C_1$ in $T$ and $\incl{S}{T}\colon S \subseteq T$. 
    From this, we obtain the instantiation $\mu_q = (j_1,\qm)$ of $C_1$ in $S$ for which $\mu_q^{T} = \mu_1$. 
    Thus, the induction hypothesis applies and 
    \begin{equation*}
        q \models d
    \end{equation*}
    for $q = \qm \circ j_1$.
    Furthermore (compare Fig.~\ref{fig:instantiation-morphism-proof} also for this), 
    \begin{align*}
        q \circ a_1	& = \qm \circ \qe \circ a_1 \\
                                & = \qm \circ b_1 \circ \gep \\
                                & = \gm \circ \gep \\
                                & = g,
    \end{align*}
    which overall implies that $g \models c$.
\end{proof}

\orderInstantiation*

\begin{proof}
    Instantiating $c$ at $\mu_0$, means to, beginning at the fixed instantiation of $C_0$ that is determined by $\mu_0$, instantiate every morphism that occurs in $c$ in every possible way that accords with the currently considered instantiation of the domain of the morphism that is to be instantiated. 
    In that, by Lemma~\ref{lem:unique-instantiation-morphism}, every instantiation of the co-domain of a morphism corresponds to exactly one instantiation of the underlying morphism, namely one where the domain is instantiated accordingly. 
    Hence, the number of possible instantiations of the co-domain provide an upper bound on how often an occurring morphism is instantiated. 
    As the instantiations of a graph $C$ in $T$ correspond to injective morphisms from $C$ to $T$ (see Lemma~\ref{lem:correspondence-instantiations}), there are maximally 
    \begin{equation*}
        \sum_{C \text{ codomain in } c}\binom{\lvert T\rvert}{\lvert C\rvert}\cdot\lvert C \rvert!
    \end{equation*}
    many (this is the case were $T$ and $C$ are sets and typing and non-existence of necessary edges do not prevent identifications). 
    Using the inequality $\binom{n}{k} \leq \frac{n^k}{k!}$ and replacing the size of the graphs from $c$ with the size of the largest one yield the claim. 
    (Note that we can ignore the fact that instantiating a negated morphism can lead to it occurring twice in the instantiated condition as this just introduces the constant factor 2.)
    
    Since the number of instantiations is finite for every morphism, the conjunctions and disjunctions in an instantiated condition are well-defined.

    The final claim follows from the fact that the maximal number of instantiations $\mu_0$ of $C_0$ in $\Sub{T}$ is also bounded by $\lvert T\rvert^{\lvert \Cmax\rvert}$. 
\end{proof}

    \fi
\end{document}